\documentclass[11pt,a4paper]{article}
\usepackage[latin1]{inputenc}
\usepackage{amsmath}
\usepackage{amsfonts}
\usepackage{amssymb}
\usepackage{amsmath, amsthm}     
\usepackage{amssymb}
\usepackage{subfigure}
\usepackage{url}
\usepackage{graphicx}
\usepackage{amsmath,amsfonts}
\usepackage{epsf,verbatim}
\usepackage[latin1]{inputenc}
\usepackage[T1]{fontenc}
\usepackage{lmodern}
\author{Farzad Aryan}

\newtheorem{lemma}{Lemma}[section]

\newtheorem{theorem}{Theorem}[section]
\newtheorem*{conjecture}{Conjecture}
\newtheorem*{conj}{Conjecture}
\theoremstyle{remark}
\newtheorem{remark}{Remark}[section]
\numberwithin{equation}{section}
\numberwithin{table}{section}
\numberwithin{figure}{section}

\begin{document}


\pagenumbering{arabic}

\date{\vspace{-2ex}}

\title{{\textbf {The distribution of k-tuples of reduced residues}}}
\author{Farzad Aryan}
\maketitle
\begin{abstract}

In 1940 Paul Erd\H{o}s made a conjecture about the distribution of reduced residues. Here we study the distribution of $k$-tuples of reduced residues.
\end{abstract}
\noindent
\\
\\
\\

\section*{Introduction}
In 1936 Cramer \cite{C}, assuming the Riemann hypothesis (RH), showed that
\begin{equation}
\label{eq-1}
\sum_{p_n <x}(p_{n+1}-p_n)^2 \ll x(\log x)^{3+\epsilon}
\end{equation} from which he deduced $p_{n+1}-p_n =O(\sqrt{p_n}\log p_n).$ Based on his probabilistic model for the primes he also conjectured that $$\limsup_{n \rightarrow \infty} \frac{p_{n+1}-p_n}{(\log p_n)^2}= 1.$$ Taking into account various sieve estimates in Cramer's probabilistic model, Granville \cite{A} in 1995 conjectured that $$\limsup_{n \rightarrow \infty} \frac{p_{n+1}-p_n}{(\log p_n)^2} \geq 2e^{-\gamma },$$
which is bigger than $1$. Note that $\gamma$ is the Euler constant. Proving \eqref{eq-1} unconditionally seems quite deep, which led P. Erd\H{o}s to make an analogous conjecture: \\
\\
\begin{conj}[Erd\H{o}s \cite{er}]
\noindent Let $q$ be a natural number, and let $P = \phi(q)/q$ be the probability that a randomly chosen
integer is relatively prime to $q$. Let $$1 = a_1 < a_2 <\cdots$$ be the integers co-prime to $q$ in increasing
order, and let $$V_\lambda (q) = \sum_{i=1}^{\phi(q)} (a_{i+1}- a_i )^ \lambda .$$ Then $$V_2 (q)\ll \phi(q)P^{-2}=qP^{-1}.$$\\
\end{conj}

\noindent More generally, the conjecture is
\begin{equation}
\label{eq-.5}
V_{\lambda }(q)  \ll qP^{1-\lambda } \text{  for $\lambda >0$.}
\end{equation}
\noindent For $\lambda  < 2$ this  was first stated in \cite{ho} and the general case was first stated in \cite{mv}. 
 Hooley \cite{ho} showed \eqref{eq-.5} for $\lambda  < 2$ .  Hausman and Shapiro \cite{sh} gave weaker upper bounds for $V_2$. Finally Montgomery and Vaughan \cite{mv} in 1986 proved \eqref{eq-.5}.\\
\\
\noindent  In this paper we investigate the distribution of of s-tuples of reduced residues which in some sense are similar
 to s-tuples of primes and we prove the analogy of Erd\H{o}s's conjecture for $s$-tuple reduced residues.\\

\noindent Let $\mathcal{D}=\lbrace h_1, h_2 , \cdots, h_s \rbrace$ and $\nu_p(\mathcal{D})$ be the number of distinct elements in $\mathcal{D}$ mod $p$. $\mathcal{D}$ is called \emph{admissible} if $\nu_p(\mathcal{D})<p$ for all primes $p$. We call $a+h_1,\ldots, a+h_s$ an \emph{ $s$-tuple of reduced residues} if they are each coprime with $q$.
\begin{theorem}
\label{Theorem}
Let $q$ be a square-free number and $\mathcal{D}=\lbrace h_1, h_2 , \cdots, h_s \rbrace$ be a fixed admissible set of integers. Let $ a_1 < a_2 <\cdots$
be those integers for which $a_i+h_1,\ldots, a_i+h_s$ is an  $s$-tuple of reduced residues. Then
$$V^{\mathcal{D}}_{\lambda }(q):=\sum_{i=1}^{\phi_{_{\mathcal{D}}}(q)} (a_{i+1}- a_i )^ \lambda  \ll \phi_{_{\mathcal{D}}}(q)P^{-s\lambda } $$
where $\phi_{_{\mathcal{D}}}(q):=\prod_{p|q}(p-\nu_p(\mathcal{D}))$, and the implied constant depends on $\mathcal{D}$ and $\lambda$.
\end{theorem}
\noindent \\
\\
\noindent The theorem follows immediately for $q$ non-square-free as well,  by considering the result for $Q=\prod_{p|q} p$. The proof of Theorem \ref{Theorem} is based on the ideas and techniques from Montgomery and Vaughan's work on the distribution of reduced residues \cite{mv}. Motivated by Theorem \ref{Theorem} the analogy of this result for primes is:\\
\begin{conjecture}
Let $p_1, \cdots$ be the set of primes for which $p_i+h_j$ are prime for all $h_j \in \mathcal{D}$. We have
$$\sum_{p_n <x}(p_{n+1}-p_n)^\lambda  \ll_{\mathcal{D}} x(\log x)^{s(\lambda -1)+\epsilon}.$$
\end{conjecture}
\noindent \\

\section*{Acknowledgments}

I am grateful to my advisor Andrew Granville for his support and helpful comments. Additionally, I would like to thank Dimitris Koukoulopoulos and Vorrapan Chandee for their careful analysis of my paper and their useful suggestions. Special appreciation for the Universit\'e de Montr\'eal's staff, my colleagues and friends, Mohammad Bardestani, Francois Charette, Dimitri Dias, Daniel Fiorilli, Kevin Henriot, and Marzieh Mehdizadeh, for their support. Finally, I would like to dedicate this paper in loving memory of my father, Yahya Aryan, whom I lost last spring.

\section{An exponential sum estimate}
In this section we prove a preliminary estimate about the distribution of $s$-tuples of reduced residues, using exponential sums. The estimate we derive here is valid for every choice of $q$, but this estimate is not the best we will give. We will prove a better estimate, using this exponential sum estimate, in section 3. \\

\begin{lemma}
Define $k_q(m)$ as follows:
\begin{equation*}
k_q(m)=
\begin{cases}
1 &\text{if $\gcd(m,q)=1$},\\
0 &\text{otherwise.}
\end{cases}
\end{equation*}
Then we have
$$
k_q(m)=P\sum_{r \mid q}\bigg( \sum_{\substack{0\leq a<r \\ (a,r)=1}}e\bigg(m\frac{a}{r}\bigg)\bigg)\frac{\mu(r)}{\phi(r)}
$$
\end{lemma}

\begin{proof}
We have
\begin{equation*}
k_q(m)=\sum_{s\mid(m,q)}\mu(s)=\sum_{s\mid q}\frac{\mu(s)}{s}\sum_{0\leq b < s}e\bigg(m\frac{b}{s}\bigg),
\end{equation*}
therefore
$$
k_q(m)= \sum_{r \mid q}\bigg( \sum_{\substack{0< a \leq r \\ (a, r)=1}}e\bigg(m\frac{a}{r}\bigg)\bigg)\bigg( \sum_{\substack{s\\r\mid s\mid q}} \frac{\mu(s)}{s}\bigg).
$$
Since
$$
 \sum_{\substack{s\\r\mid s\mid q}} \frac{\mu(s)}{s}=P\frac{\mu(r)}{\phi(r)},
$$
we can deduce
$$
k_q(m)=P\sum_{r \mid q}\bigg( \sum_{\substack{0< a \leq r \\ (a, r)=1}}e\bigg(m\frac{a}{r}\bigg)\bigg)\frac{\mu(r)}{\phi(r)}.
$$
This completes the proof of the Lemma.
\end{proof}

\noindent \\ \\
\begin{remark}
It is important to note that $\nu_p(\mathcal{D}) \leq s$ with equality if $p>h_s-h_1$. Also, if $\mathcal{D}$ is admissible, then $$\frac{1}{p} \leq 1-\frac{\nu_p(\mathcal{D})}{p} \leq 1-\frac{1}{p}$$ and we have that $$\prod_{p\leq h_s-h_1 }\frac{1}{p} \leq \prod_{p\leq h_s-h_1 } \Big(1-\frac{\nu_p(\mathcal{D})}{p}\Big) \leq \prod_{p\leq h_s-h_1 } \Big(1-\frac{1}{p}\Big).$$ Since $h_s$ and $h_1$ are fixed integers, we therefore have
$$
\prod_{\substack{p\leq h_s-h_1 \\ p\mid q}  } \Big(1-\frac{\nu_p(\mathcal{D})}{p}\Big) \asymp_{_\mathcal{D}} \prod_{\substack{p\leq h_s-h_1 \\ p\mid q}  } \Big(1-\frac{1}{p}\Big)^s.
$$
Moreover, if $p> h_s-h_1$ then $1-\frac{\nu_p(\mathcal{D})}{p} = 1-\frac{s}{p}$, so that 
$$
\prod_{\substack{p> h_s-h_1 \\ p\mid q}  } \Big(1-\frac{\nu_p(\mathcal{D})}{p}\Big) =
\prod_{\substack{p> h_s-h_1 \\ p\mid q}  } \Big(1-\frac{s}{p}\Big) \asymp_{_\mathcal{D}} \prod_{\substack{p> h_s-h_1 \\ p\mid q}  } \Big(1-\frac{1}{p}\Big)^s.
$$
Putting these together we deduce that
$$
\frac{\phi_{\mathcal{D}}(q)}{q} \asymp_{\mathcal{D}} \Big(\frac{\phi(q)}{q}\Big)^s=P^s.
$$
\end{remark}
\noindent Now we state the Lemma which we will prove at the end of this section:
\begin{lemma}
\label{th0}
Let
\begin{align}
\notag  M^{\mathcal{D}}_k(q, h)=  \sum_{n=0}^{q-1}\left( \sum_{m=1}^{h}k_q(n+m+h_1)\cdots k_q(n+m+h_s)- h\prod_{p|q} \bigg(1-\frac{\nu_p(D)}{p}\bigg) \right)^{k}.
\end{align}
Then we have that
$$ M^{\mathcal{D}}_k(q, h) \ll  qh^{k/2} P^{-2^{ks}+ks},$$
where the implicit constant depends on $k$ and $s$.
\end{lemma}

\noindent In order to go toward the proof we use exponential sums to better understand the admissible set $ \mathcal{D}=\lbrace h_1, h_2 , \cdots, h_s \rbrace.$ Also we need to prove some lemmas. We have that
$$
k_q(m)=P\sum_{r \mid q}\bigg( \sum_{\substack{0< a \leq r \\ (a, r)=1}}e\bigg(m\frac{a}{r}\bigg)\bigg)\frac{\mu(r)}{\phi(r)}.
$$
by lemma 1.1. Thus, $$
k_q(m+h_1)=P\sum_{r \mid q}\bigg( \sum_{\substack{0< a \leq r \\ (a, r)=1}}e\bigg(m\frac{a}{r}+h_1\frac{a}{r}\bigg)\bigg)\frac{\mu(r)}{\phi(r)},$$
  $\hspace{73 mm}$ 
 .
  $$\hspace{22 mm}.$$  $$\hspace{22 mm}.$$  \\
 $$k_q(m+h_s)=P\sum_{r \mid q}\bigg( \sum_{\substack{0< a \leq r \\ (a, r)=1}}e\bigg(m\frac{a}{r}+h_s\frac{a}{r}\bigg)\bigg)\frac{\mu(r)}{\phi(r)}.$$
\\
\noindent So we deduce that
\begin{align}
\label{eq-0.75}
& k_q(m+h_1)\cdots k_q(m+h_s)\\ &  \notag  =P^s\sum_{r_1, r_2, \cdots , r_s \mid q} \frac{\mu(r_1) \cdots\mu(r_s)}{\phi(r_1)\cdots \phi(r_s)}\sum_{\substack{0< a_i  \leq r_i \\ (a_i, r_i)=1 \\ 1 \leq i\leq s}} e\bigg(m\sum_{i=1}^{s}\frac{a_i}{r_i}\bigg)e\bigg(\sum_{i=1}^{s} h_i\frac{a_i}{r_i}\bigg).
\end{align}
 By summing the left-hand side of \eqref{eq-0.75} we have
 \begin{align*}
 &\sum_{m=1}^{h}k_q(n+m+h_1)\cdots k_q(n+m+h_s)\\& =  P^s\sum_{r_1, r_2, \cdots , r_s \mid q} \frac{\mu(r_1) \cdots\mu(r_s)}{\phi(r_1)\cdots \phi(r_s)}\sum_{\substack{0< a_i  \leq r_i \\ (a_i, r_i)=1 \\ 1 \leq i\leq s}} \left( \sum_{m=1}^{h}  e\bigg(m\sum_{i=1}^{s} \frac{a_i}{r_i}\bigg)e\bigg(\sum_{i=1}^{s} h_i\frac{a_i}{r_i}\bigg) \right)  e\bigg(n\bigg(\sum_{i=1}^{s} \frac{a_i}{r_i}\bigg)\bigg)\\& = P^s\sum_{r_1, r_2, \cdots , r_s \mid q} \frac{\mu(r_1) \cdots\mu(r_s)}{\phi(r_1)\cdots \phi(r_s)}\sum_{\substack{0< a_i  \leq r_i \\ (a_i, r_i)=1 \\ 1 \leq i\leq s}} \left(  E_h\bigg(\sum_{i=1}^{s} \frac{a_i}{r_i}\bigg)e\bigg(\sum_{i=1}^{s} h_i\frac{a_i}{r_i}\bigg) \right)  e\bigg(n\bigg(\sum_{i=1}^{s} \frac{a_i}{r_i}\bigg)\bigg),
 \end{align*}

\noindent where
\begin{align}
 \notag & E_h(x)=\sum_{m=1}^{h} e(mx).
 \end{align}
 To proceed with the argument we have to consider the case $\sum_{i=1}^{s} \frac{a_i}{r_i} \in \mathbb{Z}$ to extract the main term from the sum.
 We have that
 \begin{align}
 \label{eq-0.5}
 \notag & P^s\sum_{r_1, r_2, \cdots , r_s \mid q} \frac{\mu(r_1) \cdots\mu(r_s)}{\phi(r_1)\cdots \phi(r_s)}\sum_{\substack{0< a_i  \leq r_i \\ (a_i, r_i)=1 \\ 1 \leq i\leq s \\ \sum_{i=1}^{s} \frac{a_i}{r_i} \in \mathbb{Z}}} \left(  E_h\bigg(\sum_{i=1}^{s} \frac{a_i}{r_i}\bigg)e\bigg(\sum_{i=1}^{s} h_i\frac{a_i}{r_i}\bigg) \right)  e\bigg(n\bigg(\sum_{i=1}^{s} \frac{a_i}{r_i}\bigg)\bigg) \\& = hP^s\sum_{r_1, r_2, \cdots , r_s \mid q} \frac{\mu(r_1) \cdots\mu(r_s)}{\phi(r_1)\cdots \phi(r_s)}\sum_{\substack{0< a_i  \leq r_i \\ (a_i, r_i)=1 \\ 1 \leq i\leq s \\ \sum_{i=1}^{s} \frac{a_i}{r_i} \in \mathbb{Z} }}e\bigg(\sum_{i=1}^{s} h_i\frac{a_i}{r_i}\bigg),
 \end{align}
since $E_h(r)=h$ for all integers $r$. Now, we need to use Lemma 3 of \cite{MS} (due to Hardy and Littlewood). Hardy and Littlewood proved that $$ \mathfrak{S}_q(D)=  \sum_{r_1, r_2, \cdots , r_s |q} \frac{\mu(r_1) \cdots\mu(r_s)}{\phi(r_1)\cdots \phi(r_s)}\sum_{\substack{0< a_i  \leq r_i \\ (a_i, r_i)=1 \\ 1 \leq i\leq s \\\sum_{i=1}^{s} \frac{a_i}{r_i} \in \mathbb{Z} }}e\bigg(\sum_{i=1}^{s} h_i\frac{a_i}{r_i}\bigg)
 $$
 where  $\mathfrak{S}$ is the singular series $$\mathfrak{S}_q(D)=\prod_{p|q} \bigg(1-\frac{1}{p}\bigg)^{-s} \bigg(1-\frac{\nu_p(\mathcal{D})}{p}\bigg).$$
\noindent Using this we have
 \begin{align*}
 &\sum_{m=1}^{h}k_q(n+m+h_1)\cdots k_q(n+m+h_s)- h\prod_{p|q} \bigg(1-\frac{\nu_p(\mathcal{D} )}{p}\bigg) \\& =P^s\sum_{r_1, r_2, \cdots , r_s \mid q} \frac{\mu(r_1) \cdots\mu(r_s)}{\phi(r_1)\cdots \phi(r_s)}\sum_{\substack{0< a_i  \leq r_i \\ (a_i, r_i)=1 \\ 1 \leq i\leq s  \\ \sum_{i=1}^{s} \frac{a_i}{r_i} \notin \mathbb{Z} }} \left(  E_h\bigg(\sum_{i=1}^{s} \frac{a_i}{r_i}\bigg)e\bigg(\sum_{i=1}^{s} h_i\frac{a_i}{r_i}\bigg) \right)  e\bigg(n\bigg(\sum_{i=1}^{s} \frac{a_i}{r_i}\bigg)\bigg)
 \end{align*}
and, consequently,
\begin{align}
\label{eq-0.25}
 &\ \left( \sum_{m=1}^{h}k_q(n+m+h_1)\cdots k_q(n+m+h_s)- h\prod_{p|q} \bigg(1-\frac{\nu_p(\mathcal{D} )}{p}\bigg) \right)^{k} \\&= \notag P^{ks} \sum_{ \substack{r_{i, j} \mid q \\ 1 \leq i\leq k \\ 1 \leq j \leq s }} \left( \prod_{i, j} \frac{\mu(r_{i,j})}{\phi(r_{i,j})}\right) \sum_{\substack{0 < a_{i,j} \leq r_{i,j} \\ (a_{i,j}, r_{i,j})=1 \\ 1\leq j\leq s  \\ \sum_{j=1}^{s} \frac{a_{i,j}}{r_{i,j}} \notin \mathbb{Z} \\ 1\leq i \leq k }} \left(  E_h\bigg(\sum_{i=1}^{s} \frac{a_{1,j}}{r_{1,j}}\bigg)\cdots E_h\bigg(\sum_{j=1}^{s} \frac{a_{k,j}}{r_{k,j}}\bigg)e\bigg(\sum_{i,j} h_j\frac{a_{i,j}}{r_{i,j}}\bigg) \right) \\& \notag \quad \quad \times  e\bigg(n\bigg(\sum_{i,j} \frac{a_{i,j}}{r_{i,j}}\bigg)\bigg).
\end{align}
Summing \eqref{eq-0.25} over $n$ mod $q$ and using the fact that when $q\sum_{i}\rho_i \in \mathbb{Z}$  $$\sum_{n=0}^{q-1} e\big(n\big(\sum_{i}\rho_i\big)\big)=0$$ unless $\sum_{i} \rho_i \in \mathbb{Z}$, we have that
\begin{align*}
 &\ \sum_{n=0}^{q-1}\left( \sum_{m=1}^{h}k_q(n+m+h_1)\cdots k_q(n+m+h_s)- h\prod_{p|q} \bigg(1-\frac{\nu_p(\mathcal{D})}{p}\bigg) \right)^{k}  \\&= qP^{ks} \sum_{ \substack{r_{i, j} \mid q \\ 1 \leq i\leq k \\ 1 \leq j \leq s }} \left( \prod \frac{\mu(r_{i,j})}{\phi(r_{i,j})}\right) \sum_{\substack{ 1\leq i \leq k \\ 0 < a_{i,j} \leq r_{i,j} \\ (a_{i,j}, r_{i,j})=1   \\ \sum_{j=1}^{s} \frac{a_{i,j}}{r_{i,j}} \notin \mathbb{Z}\\ \sum_{i,j} \frac{a_{i,j}}{r_{i,j}} \in \mathbb{Z} }} \left(  E_h\bigg(\sum_{j=1}^{s} \frac{a_{1,j}}{r_{1,j}}\bigg)\cdots E_h\bigg(\sum_{j=1}^{s} \frac{a_{k,j}}{r_{k,j}}\bigg)e\bigg(\sum_{i,j} h_j\frac{a_{i,j}}{r_{i,j}}\bigg) \right).
\end{align*}
Let $F(x)=\min(h, \frac{1}{\Vert x \Vert})$ where $\Vert x \Vert$ is the distance between $x$ and the closest integer to $x$. We have that $|E_h(x)|\leq F(x)$, and consequently 
\begin{align}
\label{eq0}
\notag &\ \sum_{n=0}^{q-1}\left( \sum_{m=1}^{h}k_q(n+m+h_1)\cdots k_q(n+m+h_s)- h\prod_{p|q} \bigg(1-\frac{\nu_p(D)}{p}\bigg) \right)^{k}   \notag \\& \ll qP^{ks} \sum_{ \substack{\textbf{r} \mid q }}  \sum_{  [r_{1,1}, r_{1,2}, \cdots r_{k,s}]=\textbf{r} } \frac{S(\{r_{i, j}\}_{i, j})}{\left( \prod \phi(r_{i,j}) \right)}
\end{align}
where $$S(\{r_{i, j}\}_{i, j})= \sum_{\substack{ 0 < a_{i,j} \leq r_{i,j} \\ (a_{i,j}, r_{i,j})=1   \\ \sum_{j=1}^{s} \frac{a_{i,j}}{r_{i,j}} \notin \mathbb{Z}\\ \sum_{i,j} \frac{a_{i,j}}{r_{i,j}} \in \mathbb{Z} }} F\bigg(\sum_{j=1}^{s} \frac{a_{1,j}}{r_{1,j}}\bigg)\cdots F\bigg(\sum_{j=1}^{s} \frac{a_{k,j}}{r_{k,j}}\bigg)$$

\begin{lemma}
\label{lemma2}
Every element of the form
$$ \sum_{\substack{j=1 }}^{s}  \frac{a_{i,j}}{r_{i,j}} \hspace{2 mm}   \hspace{2 mm} where \hspace{2 mm} \hspace{1 mm} 0 < a_{i,j} \leq r_{i,j}   $$ can be written as $$ \frac{a}{[r_{i, 1}, r_{i, 2}, \cdots r_{i, s}]} (\;\mbox {{\rm mod} $ 1 )$} , \quad {where} \quad \hspace{1 mm} 1\leq a \leq [r_{i, 1}, r_{i, 2}, \cdots r_{i, s}],$$ and each fraction that has such a representation has exactly
$\frac{r_{i, 1} r_{i, 2} \cdots r_{i, s}}{[r_{i, 1}, r_{i, 2}, \cdots r_{i, s}]}$ representations.
\end{lemma}

\noindent By $\frac{r_{i, 1} r_{i, 2} \cdots r_{i, s}}{[r_{i, 1}, r_{i, 2}, \cdots r_{i, s}]}$ representations we mean that
the equation $$ \sum_{\substack{j=1 }}^{s}  \frac{a_{i,j}}{r_{i,j}} =  \tau \;\mbox{(mod $ 1 $)}$$ has exactly $\frac{r_{i, 1} r_{i, 2} \cdots r_{i, s}}{[r_{i, 1}, r_{i, 2}, \cdots r_{i, s}]}$ different solutions, if it has any.
\begin{proof}
Let $d=(r_1, r_2)$ and we call $r_i^{\prime}=\frac{r_i}{d}$ for $i=1, 2$. For fixed $a, b$ we are interested in the number of solutions for the equation
$$ \frac{a}{r_1}+\frac{b}{r_2}  =  \frac{x}{r_1}+\frac{y}{r_2}  \;\mbox{(mod $ 1 $)} $$
where $1\leq x\leq r_1$ and $1\leq y\leq r_2$, which leads us to the number of solutions of
\begin{equation}
\label{eq1}
ar_2^{\prime} + br_1^{\prime} \equiv xr_2^{\prime} + yr_1^{\prime} \;\mbox{(mod $r_1^{\prime}r_2^{\prime}d$)}.
\end{equation}
 We have $a \equiv x \; \mbox{ (mod $r^{\prime}_1$)} $ and $b \equiv y \; \mbox{ (mod $r^{\prime}_2$)} $. Let $x=a+ir_1^{\prime}$ and $y=b+jr_2^{\prime}.$ Then by using \eqref{eq1} we have
 $$
(a-x)r_2^{\prime}  \equiv  (y-b)r_1^{\prime} \;\mbox{(mod $r_1^{\prime}r_2^{\prime}d$)}.$$ Therefore we have $i  +j \equiv 0 \;\mbox{(mod $d$)}$,  which has exactly $d$ solutions. So we conclude that, given $a$ and $b$, there are exactly $d$ solutions $(x, y)$ with $1\leq x \leq r_1$ and $1\leq y \leq r_2$ to the equation
$$ \frac{a}{r_1}+\frac{b}{r_2}  =  \frac{x}{r_1}+\frac{y}{r_2} \;\mbox{(mod $ \mathbb{Z} $)}.$$
Obviously $\frac{a}{r_1}+\frac{b}{r_2} \;\mbox{(mod $ 1 $)} \in \left\lbrace  \frac{t}{[r_1, r_2]} : 0 \leq t \leq\ [r_1, r_2] \right\rbrace   $ and as we showed above, each element is repeated exactly  $d=\frac{r_1r_2}{[r_1, r_2]}$ times. This proves the lemma for $s=2$. Using induction, we have that $$\frac{a_{i,1}}{r_{i,1}}+ \cdots +\frac{a_{i,k-1}}{r_{i,k-1}} = \frac{a}{[r_{i, 1}, r_{i, 2}, \cdots r_{i, k-1}]},  $$ with exactly $\frac{r_{i, 1} r_{i, 2} \cdots r_{i, k-1}}{[r_{i, 1}, r_{i, 2}, \cdots r_{i, k-1}]}$ repetitions each. And, by the first part of the proof there are exactly $$\frac{[r_{i, 1}, r_{i, 2}, \cdots r_{i, k-1}]r_{i,k}}{[r_{i, 1}, r_{i, 2}, \cdots r_{i, k}]}$$ ways to write $\frac{a_{i,1}}{r_{i,1}}+ \cdots +\frac{a_{i,k}}{r_{i,k}}$ as $\frac{a}{[r_{i, 1}, r_{i, 2}, \cdots r_{i, k-1}]} + \frac{a_{i,k}}{r_{i,k}} \quad \text{(mod} 1 \text{)}.$  Now the total number of repetitions is $$\frac{[r_{i, 1}, r_{i, 2}, \cdots r_{i, k-1}]r_{i,k}}{[r_{i, 1}, r_{i, 2}, \cdots r_{i, k}]}\cdot \frac{r_{i, 1} r_{i, 2} \cdots r_{i, k-1}}{[r_{i, 1}, r_{i, 2}, \cdots r_{i, k-1}]}= \frac{r_{i, 1} r_{i, 2} \cdots r_{i, k}}{[r_{i, 1}, r_{i, 2}, \cdots r_{i, k}]} $$
\end{proof}
Now our task is to bound \eqref{eq0}, for which we need to use the idea of Montgomery and Vaughan's Fundamental Lemma \cite{mv}, slightly modified. In order to do that we use Lemma 2 from \cite[Page 596]{MS}.

\begin{lemma}
\label{lemma3}
 Let $q_1,\cdots , q_k$ be square-free integers, each one strictly greater than 1, and
put $d = [q_1, . . . , q_k]$. Let $G$ be a complex-valued function defined on $(0, 1)$, and suppose
that $G_0$ is a nondecreasing function on the positive integers such that
$$\sum_{a=1}^{q-1} |G(a/q)|^2 \leq qG_0(q),$$
for all square-free integers $q > 1$. Then
 $$\bigg| \sum_{\substack{a_1, \cdots, a_k \\ 0<a_i<q_i \\ \sum \frac{a_i}{q_i} \in \mathbb{Z}}} \prod_{i=1}^{k} G(a_{i}/q_i) \bigg | \leq \frac{1}{d} \prod_{i=1}^{k} q_iG_0(q_i)^{1/2}.$$
\end{lemma}
We now need to verify that  $F$ satisfies  the requirements for $G$ in Lemma \ref{lemma3}. Lemma 4 of \cite{mv} asserts that
$$\sum_{0<a<q}F\bigg(\frac{a}{q}\bigg)^2 \ll q\min(q, h).$$ Since $\min(q, h)$ is obviously a non-decreasing function of $q$, we can use Lemma \ref{lemma3} with $F$ and $\min(q, h)$ in place of $G$ and $G_0$ respectively. About the condition $q_i>1$, note that, since we apply Lemma \ref{lemma3} for $q_i=[r_{i,1}, \cdots, r_{i,s}]$ and we have $\sum_{j=1}^{s} \frac{a_{i,j}}{r_{i,j}} \notin \mathbb{Z}$, then $q_i=[r_{i,1}, \cdots, r_{i,s}] \neq 1$. From Lemma \ref{lemma2} we have
\begin{align}
\notag S(\{r_{i, j}\}_{i, j})= \sum_{\substack{0 < a_{i,j} \leq r_{i,j}  \\ (a_{i,j}, r_{i,j})=1   \\ \sum_{j=1}^{s} \frac{a_{i,j}}{r_{i,j}} \notin \mathbb{Z}\\ \sum_{i,j} \frac{a_{i,j}}{r_{i,j}} \in \mathbb{Z} }} F\bigg(\sum_{j=1}^{s} \frac{a_{1,j}}{r_{1,j}}\bigg)\cdots F\bigg(\sum_{j=1}^{s} \frac{a_{k,j}}{r_{k,j}}\bigg)\\ \leq  T\sum_{\substack{0<a_{i}<[r_{i,1}, \cdots, r_{i,s}]\\ \sum_{i=1}^{s} \frac{a_{i}}{{[r_{i,1}, \cdots, r_{i,s}]}} \in \mathbb{Z} }} F\bigg(\frac{a_{1}}{[r_{1,1}, \cdots, r_{1,s}]}\bigg)\cdots F\bigg(\frac{a_{k}}{[r_{k,1}, \cdots, r_{k,s}]}\bigg),
\end{align}
where $$T=\frac{r_{1,1}\cdots r_{1,s}}{[r_{1,1}, \cdots , r_{1,s}]} \cdots \frac{r_{k,1}\cdots r_{k,s}}{[r_{k,1}, \cdots , r_{k,s}]}. $$
\\
\noindent Now using Lemma \ref{lemma3} with $G=F$ and $q_i=[r_{i,1}, \cdots, r_{i,s}],$ we have that
 \begin{equation}
\label{eq3}
S(\{r_{i, j}\}_{i, j})\ll \frac{r_{1,1}\cdots r_{k,s}}{\textbf{r}}h^{k/2}
\end{equation}
We are now ready to prove Lemma \ref{th0}:
\begin{proof}[Proof of Lemma \ref{th0}]
We prove the result with $q$ square-free. Then the Lemma follows for $q$ non-square-free immediately by considering the result for $Q=\prod_{p|q} p$. Now using relations \eqref{eq0} and \eqref{eq3}, we have that
\begin{align}
\notag &\ \sum_{n=0}^{q-1}\left( \sum_{m=1}^{h}k_q(n+m+h_1)\cdots k_q(n+m+h_s)- h\prod_{p|q} \bigg(1-\frac{\nu_p(D)}{p}\bigg) \right)^{k}  \\& \ll  \notag  qP^{ks}\sum_{r|q} \frac{1}{r}\sum_{\substack{[r_{1,1}, r_{1,2}, \cdots,\hspace{1 mm} r_{k,s}]=r}} \frac{r_{1,1}\cdots r_{k,s}}{\phi(r_{1,1}) \cdots \phi(r_{k,s})}h^{k/2} \\ \leq \notag & qP^{ks}\sum_{r|q} \frac{1}{r}\bigg(\sum_{\substack{r^{\prime} | r}} \frac{r^{\prime}}{\phi(r^{\prime})}\bigg)^{ks}h^{k/2} =qh^{k/2} P^{ks} \prod_{p|q} \left(1+\frac{1}{p}\Big(2+\frac{1}{p-1}\Big)^{ks} \right)  \\& \notag \ll qh^{k/2} P^{-2^{ks}+ks}.
\end{align}
\end{proof}

\noindent \\

\section{A probabilistic estimate}
In this section we prove an estimate about the distribution of $s$-tuples of reduced residues using a probabilistic method. The estimate derived here is valid only when $q$ is not divisible by any small prime, and in this case it is the best possible we can have. In particular, it's much better than our earlier exponential sum estimate in this range. \\

\noindent Let $X_i$, for $1\leq i \leq h$, be independent identically distributed random variables such that $$\rm{Prob}(X_i=1)=1-\rm{Prob}(X_i=0)=P.$$ Then $$X=X_1+ \cdots + X_h$$ is called a \emph{binomial random variable}. Given such a random variable $X$, we denote with $\mu_k(h, P)$ its $k$-th moment about its mean, that is to say,  $$\mu_k(h, P):=\mathbb{E}\big((X-hP)^{k} \big).$$ 

\begin{lemma}
\label{lemma4}
Let $A$ be a set of $h$ integers and $h_1< \cdots < h_s$. Suppose that for each prime divisor $p$ of $q$ we have $p> \max A- \min A+ h_s-h_1$. Suppose also that $p > y$ for all $p| q$. Then for $y> h^k$ and for each fixed even $k > 1$
\begin{align*}
M^{\mathcal{D}}_k(q, h)=&\sum_{n=0}^{q-1} \Bigg( \sum_{\substack{m \in A \\ (n+m+h_i, q)=1 \\ 1\leq i \leq s}} 1 -h \bigg(\frac{\phi(q)}{q}\bigg)^{s} \Bigg)^k
 \ll  q\bigg(h\bigg(\frac{\phi(q)}{q}\bigg)^s\bigg)^{[k/2]}+qh\bigg(\frac{\phi(q)}{q}\bigg)^s,
\end{align*}
which the implicit constant depends on $k$ and $|h_1-h_s|$.
\end{lemma}
\noindent
\begin{remark}
Under the conditions of Lemma \ref{lemma4}, it provides a better estimate than Lemma \ref{th0}. Indeed Lemma \ref{th0} yields the estimate $$M^{\mathcal{D}}_k(q, h)  \ll  qh^{k/2} P^{-2^{ks}+ks}$$ whereas by Lemma \ref{lemma4} we have that  $$M^{\mathcal{D}}_k(q, h) \ll  q\bigg(hP^s\bigg)^{[k/2]}+qhP^s.$$ Comparing two bounds and using the fact that $q\Big(hP^s\Big)^{[k/2]} \leq qh^{k/2} P^{-2^{ks}+ks}$ proves the point.
\end{remark}
\begin{proof}
The proof is similar to the proof of Lemma 9 in \cite{mv}, with a small variation which we explain. We have that
$$\sum_{\substack{m \in A \\ (m+h_i, q)=1 \\ 1\leq i \leq s}} 1 -h P^{s} = \sum_{1 \leq j \leq H} \Bigg(\sum_{\substack{m \in A_j \\ (m+h_i, q)=1 \\ 1\leq i \leq s}} 1 -|A_j| P^{s}\Bigg)$$

\noindent where
\begin{align*}
A_j=\left\lbrace m \in A : m \equiv j \;\mbox{(mod $H$)} \right\rbrace ,
\end{align*}
where $H= |h_s-h_1| +1$. From H$\ddot{\rm o}$lder's inequality with, $\frac{1}{k}+ \frac{1}{\frac{k}{k-1}}=1$, we have that
$$\bigg| \sum_{i=1}^{H} a_i \bigg| \leq H^{\frac{k-1}{k}} \bigg( \sum_{i=1}^{H} |a_i|^{k} \bigg)^{\frac{1}{k}} $$ and, consequently,
$$\Bigg( \sum_{\substack{m \in A \\ (m+h_i, q)=1 \\ 1\leq i \leq s}} 1 -h P^{s}\Bigg)^k  \leq H^{k-1} \sum_{1 \leq j \leq H} \Bigg( \sum_{\substack{m \in A_j \\ (m+h_i, q)=1 \\ 1\leq i \leq s}} 1 -|A_j| P^{s} \Bigg)^k. $$
Now we focus on
\begin{equation}
\label{eq2}
S_j=\sum_{n=0}^{q-1} \Bigg( \sum_{\substack{m \in A_j \\ (n+m+h_i, q)=1 \\ 1\leq i \leq s}} 1 -|A_j| P^{s} \Bigg)^k.
\end{equation}
We note that
$$S_j=\sum_{n} \sum_{r} {k\choose r}\bigg(\sum_{\substack{m \in A_j \\ (n+m+h_i, q)=1 \\ 1\leq i \leq s}} 1\bigg)^r \bigg(-|A_j|P^s \bigg)^{k-r}.$$
Moreover, we have that
$$
 \bigg(\sum_{\substack{m \in A_j \\ (n+m+h_i, q)=1 \\ 1 \leq i \leq s}} 1 \bigg)^r= \sum_{
\substack{m_1 , \cdots, m_r \in A_j \\ (n+m_l+h_i, q)=1 \\ 1 \leq i \leq s \\ 1 \leq l \leq r}} 1$$
We will show that $m_l+h_i \neq m_{l^{\prime}} + h_{i^{\prime}}$ for $m_l \neq m_{l^{\prime}}$. Without loss of generality, we assume that $m_l < m_{l^{\prime}}$ and therefore  $m_l+h_i < m_{l^{\prime}}+ h_{i^{\prime}}.$ This is true since $m_l - m_{l^{\prime}} \equiv 0 \;\mbox {{\rm (mod} $ H $}) $ and thus $|m_l-m_{l^{\prime}}|\geq H > |h_i - h_{i^{\prime}}| $. Now we claim that $m_l +h_i  \not \equiv  m_{l^{\prime}}+h_{i^{\prime}} \;\mbox {{\rm( mod} $ p $})$ for all $p|q$. Assume, on the contrary, that $$m_l +h_i  \equiv  m_{l^{\prime}}+h_{i^{\prime}} \;\mbox {({\rm mod} $ p $)}$$ for some $p|q$. Then  we have that $p| m_l+h_i - \big(m_{l^{\prime}} + h_{i^{\prime}}\big)$. We already have shown $m_l+h_i - \big(m_{l^{\prime}} + h_{i^{\prime}}\big) \neq 0$, therefore $$p \leq |m_l-m_{l^{\prime}}| + |h_i-h_{i^{\prime}}|,$$ which contradicts our assumption that $p> \max A- \min A+ h_s-h_1$.\\
\noindent Applying these facts and changing the order of summation in $S_j$, we have that
\begin{align}
\sum_{{
\substack{n=0 \\ (n+m_j+h_i, q)=1 \\ 1 \leq i \leq s \\ 1\leq j \leq r}} }^{q-1} 1 = \prod_{p|q} (p-st),
\end{align}where $t= \# \left\lbrace m_1, \cdots, m_r \right\rbrace$. Let  $S(r, t)$ denote the Stirling number of the second kind, i.e. the number of ways of partitioning a set of cardinality $r$ into exactly $t$ non-empty subsets. Following the proof of Lemma 9 in \cite{mv}, $S(r, t)t!$ is the number of surjective maps from a set of cardinality $r$ to a set of cardinality $t$. We set $S(r,0) = 0$ so that we have
$$\sum_{n=0}^{q-1}\sum_{
\substack{m_1 , \cdots, m_r \in A_j \\ (n+m_k+h_i, q)=1 \\ 1 \leq i \leq s \\ 1 \leq k \leq r}} 1= \sum_{t=0}^{r}\sum_{\substack{\mathcal{B} \subseteq A_j \\ \text{card}(\mathcal{B})=t }} S(r, t)t!\prod_{p|q}(p-st) $$  As there are ${|A_j| \choose t}$ possible choices for $\mathcal{B}$, the above is $$q\sum_{t=1}^{r}{|Aj| \choose t} S(r, t)t!P^{st}\prod_{p|q}\Big(1-\frac{st}{p}\Big)\bigg(1-\frac{1}{p}\bigg)^{-st}$$ and, since $p>y> h_s-h_1$ we have that
$$\prod_{p|q}\bigg(1-\frac{st}{p}\bigg)\bigg(1-\frac{1}{p}\bigg)^{-st}=1+ O_{st}\Big(\frac{1}{y}\Big)$$
From Lemma 9 in \cite[page.326]{mv}  we have that $$S_j=q\sum_{r=0}^{k}{k \choose r} (-|A_j|P^{s})^{k-r} \sum_{t=0}^{r} {|A_j| \choose t} S(r, t)t!(P)^{st}\Big(1+O_{st}(\frac{1}{y})\Big)$$ and $$q\sum_{r=0}^{k}{k \choose r} (-|A_j|P^{s})^{k-r} \sum_{t=0}^{r} {|A_j| \choose t} S(r, t)t!(P)^{st}= \mu_k(|A_j|, P^s)$$ using \cite[page.327]{mv}. Thus
\begin{align*}
S_j= & q\sum_{r=0}^{k}{k \choose r} \bigg(-|A_j|P^s\bigg)^{k-r} \sum_{t=0}^{r} {|A_j| \choose t} S(r, t)t!P^{st}\bigg(1+ O_{st}\Big(\frac{1}{y}\Big)\bigg) \\ & =q\mu_k\bigg(|A_j|,P^s \bigg)+ O\bigg\{\frac{q}{y}\Big(hP^s\Big)^k +hP^s\bigg\},
\end{align*}
 using the fact that $|A_j| \leq h$. For the error term the dependence of the implicit constant on $t$ can be considered to be a dependence on $k$, since $t < s < k$ we also use $$\frac{q}{y}\sum_{t=0}^{r} {|A_j| \choose t} P^{st} \ll \frac{q}{y}\Bigg((h^s)^r +hP^s\Bigg)$$\\
\noindent Next note that Lemma 11 of \cite{mv} states that, for any fixed integer $k > 0$, $\mu_{k}(h, P) \ll (hP)^{[k/2]} + hP,$ uniformly for $0 < P < 1$, $h = 1, 2, 3, ...$. So $$\mu_{k}(|A_j|, P^s) \ll (|A_j|P^s)^{[k/2]} + |A_j|P^s\leq (hP^s)^{[k/2]} + hP^s.$$ Using this and and our assumption that $y>h^k$, we find that
\begin{align}
\label{eq4}
 &\sum_{n=0}^{q-1} \Bigg( \sum_{\substack{m \in A \\ (n+m+h_i, q)=1 \\ 1\leq i \leq s}} 1 -h P^{s} \Bigg)^k
 \ll  q\Big(hP^s\Big)^{[k/2]}+qhP^s,
\end{align}
which concludes the proof of the Lemma.
\end{proof}
\noindent
\\
\\
\section{Proof of Theorem 0.1}
In this section we will prove a estimate about the distribution of $s$-tuples of reduced residues, by combining  both our probabilistic and exponential sum estimates. The new estimate that we derive here is valid for every $q$ and it is better than our exponential sum estimate. Using the this estimate, we will be able to prove Theorem \ref{Theorem}.\\

\begin{lemma}
\label{P.e} Let $k$ be a given even number, and fix constant $A>k$. Let $q_1= \prod_{\substack{p|q \\p\leq y }} p$ and $q_2= \prod_{\substack{p|q \\ p>y}} p$, where $h^{A}>y>h^k$. Correspondingly we set $P_i= \frac{\phi(q_i)}{q_i}$ for $i=1, 2$.  \it{For $h> P^{-1}$ we have}
\begin{align*}
M^{\mathcal{D}}_k(q, h) \ll q(hP^s)^{[k/2]}+qh({P)^s} + qh^{k/2}P_1^{-2^{ks}+ks}P_2^{sk}.
\end{align*}
And the implicit constant depends on $k$ and $s$.
\end{lemma}
\begin{proof}
Since $q$ is square-free we have $q = q_1 q_2$ and $(q_1, q_2) = 1$. By the Chinese Remainder Theorem we have that
$$M^{\mathcal{D}}_k(q, h)= \sum_{n_1=0}^{q_1-1}\sum_{n_2=0}^{q_2-1} D(n_1, n_2)^k, $$
where
$$D(n_1, n_2) =  \sum_{\substack{m =1 \\ (n_i+m+h_j, q_i)=1 \\ 1\leq j \leq s \\ i=1,2}}^{h} 1 -h \prod_{p|q} \bigg(1-\frac{\nu_p(\mathcal{D})}{p}\bigg).
$$
Following \cite{mv}, we may write $D=D_1+D_2$ where
$$D_1=\prod_{p|q_2} \bigg(1-\frac{\nu_p(\mathcal{D})}{p}\bigg)\sum_{\substack{m =1 \\ (n_1+m+h_j, q_1)=1 \\ 1\leq j \leq s }}^{h} 1 -h \prod_{p|q} \bigg(1-\frac{\nu_p(\mathcal{D})}{p}\bigg)$$
$$D_2=\sum_{\substack{m =1 \\ (n_i+m+h_j, q_i)=1 \\ 1\leq j \leq s \\ i=1,2}}^{h} 1-\prod_{p|q_2} \bigg(1-\frac{\nu_p(\mathcal{D})}{p}\bigg)\sum_{\substack{m =1 \\ (n_1+m+h_j, q_1)=1 \\ 1\leq j \leq s }}^{h} 1$$
From Holder's inequality we have $D^k \leq 2^k\big(D_1^{k}+D_2^{k}\big),$ and consequently
$$M^{\mathcal{D}}_k(q, h) \ll \sum_{n_1} \sum_{n_2} D_1 ^k + \sum_{n_1} \sum_{n_2}D_2^k.$$
Since $D_1$ is independent of $n_2$, we have that
\begin{equation}
\label{eq4.5}
\sum_{n_1} \sum_{n_2} D_1 ^k \ll q_2P^{sk}_2M^{\mathcal{D}}_k(q_1, h),
\end{equation}
which by Lemma \ref{th0} leads to
$$\sum_{n_1} \sum_{n_2} D_1 ^k \ll_k q_2P^{sk}_2q_1h^{k/2} P_1^{-2^{ks}+ks}=qh^{k/2}P_1^{-2^{ks}+ks}P_2^{sk}.$$
To estimate $\sum_{n_1} \sum_{n_2} D_2^k$ let

 $$A_{n_{1}}= \left\lbrace 1\leq m\leq h : (n_1+m+h_j, q_1)=1 , 1 \leq j \leq s \right\rbrace. $$
Note that the size of $A_{n_1}$ is $$\sum_{\substack{m =1 \\ (n_1+m+h_j, q_1)=1 \\ 1\leq j \leq s }}^{h} 1,$$
which, by a simple sieve argument, is $\ll hP^{s}_{1}$. Therefore
$$\sum_{n_2} D_2^k= \sum_{n_2}\bigg( \sum_{\substack{m \in A_{n_1} \\ (n_2+m+h_j, q_2)=1 \\ 1\leq j \leq s \\}} 1-\prod_{p|q_2} \bigg(1-\frac{\nu_p(\mathcal{D})}{p}\bigg)|A_{n_1}| \bigg)^k.$$ Now, since $y> h^k$ and $p|q_2$, we have that $p>h^k$ and consequently $$\prod_{p|q_2}\bigg(1-\frac{\nu_p(\mathcal{D})}{p}\bigg)= \prod_{p|q_2}\Big(1-\frac{s}{p}\Big)= \prod_{p|q_2}\Big(1-\frac{1}{p}\Big)^s \Big(1+ O\Big(\frac{1}{y}\Big)\Big).$$
Next we need to use Lemma \ref{lemma4} with $A=A_{n_1}$ and $q=q_2$. In order to do this we need to verify that $p> \max A_{n_1} - \min A_{n_1} + h_s- h_1$ for all $p|q_2$. We have that $\max A_{n_1} - \min A_{n_1} \leq h$ and, since for $p| q_2 $ we have $p>y > h^k$, it suffices to verify that  $h+H< h^k$. This is true because $H$ is fixed, $k\geq2$ and $h>P^{-1}$. (Note that we may assume $P^{-1} >H$ else, otherwise, $P^{-1}$ is bounded and we can deduce the result desired here   from Lemma \ref{th0}.) Using Lemma \ref{lemma4}, we have that
 $$\sum_{n_2} D_2^k \ll q_2\bigg(|A_{n_{1}}|\bigg(\frac{\phi(q_2)}{q_2}\bigg)^s\bigg)^{[k/2]}+q_2|A_{n_{1}}|\bigg(\frac{\phi(q_2)}{q_2}\bigg)^s.$$ Since $|A_{n_{1}}| \ll hP^{s}_{1}$, we have that
 $$\sum_{n_2} D_2^k \ll q_2\bigg(hP^s\bigg)^{[k/2]}+q_2hP^s, $$
consequently we have that \big(with $P=\frac{\phi(q)}{q}$\big)
$$\sum_{n_1} \sum_{n_2} D_2^k \ll q(hP^s)^{[k/2]}+qhP^s. $$
Finally, we arrive at our desired estimate
\begin{align}
\label{P.E}
M^{\mathcal{D}}_k(q, h) \ll q(hP^s)^{[k/2]}+qh({P)^s} + qh^{k/2}P_1^{-2^{ks}+ks}P_2^{sk}.
\end{align}
\end{proof}
Now, Using above estimate we will prove Theorem \ref{Theorem}. Let $a_1< a_2 < \cdots$ be the integers, such that $a_i +h_j$ is co-prime to $q$ for each $h_j \in \mathcal{D}$. Let
$$L(x)=\# \left\lbrace i : 1 \leq i \leq \phi_{_{\mathcal{D}}}(q), a_{i+1}-a_{i}> x  \right\rbrace.$$
We have that $$V^{\mathcal{D}}_{\lambda }(q)= \lambda \int_{0}^{\infty} L(x)x^{\lambda -1}dx.$$
Obviously, $L(x) \leq \prod_{p|q} (p-\nu_p(\mathcal{D})) < CqP^s$ for some constant $C=C(\mathcal{D})$. Therefore for $x<P_{_{\mathcal{D}}}^{-1}$ \Big(with $P_{_{\mathcal{D}}}=\prod_{p|q} \Big(1-\frac{\nu_p(\mathcal{D})}{p}\Big)$\Big), since $P_{_{\mathcal{D}}}\asymp P^{s}$, we have that
$$\lambda \int_{0}^{P_{_{\mathcal{D}}}^{-1}} L(x)x^{\lambda -1} \ll qP^s \int_{0}^{P_{_{\mathcal{D}}}^{-1}} x^{\lambda -1} \ll q(P^{s})^{1-\lambda }. $$
To bound $L(x)$ for larger $x$, we note that if $a_{i+1} - a_i > h$, for some integer $h$. Then
$$\sum_{\substack{m =1 \\ (n+m+h_j, q)=1 \\ 1\leq j \leq s }}^{h} 1 -hP_{_{\mathcal{D}}}= -hP_{_{\mathcal{D}}} $$ for $a_i \leq n < a_{i+1} - h$. Let $k$ be a fixed even integer bigger than $2\lambda $. Then
\begin{align}
\label{eq5}
\sum_{\substack{i=1 \\ a_{i+1}-a_i >h}}^{qP_{_{\mathcal{D}}}} (a_{i+1}-a_i- h)(hP_{_{\mathcal{D}}})^{k} \leq M^{\mathcal{D}}_k(q, h).
\end{align}
If $h=[\frac{x}{2}]$ and $a_{i+1}-a_i > x$, then $a_{i+1}-a_i-h > h$, so the left-hand side of \eqref{eq5} is $\geq h(hP_{_{\mathcal{D}}})^{k}L(x)$. Combining this with our estimate in Lemma \ref{P.e} yields
$$x(xP_{_{\mathcal{D}}})^k L(x) \ll q\big((xP_{_{\mathcal{D}}})^{k/2} + x^{k/2}P^{ks} P^{-2^{ks}}_1\big).
$$Now for $x<e^{P^{-\alpha}} $, if $y=x^{2k}+1$ where, $\alpha= \frac{ks}{2^{ks+1}}$, then we have $$P^{-1}_1=\big( \prod_{p< y}(1-\frac{1}{p})\big)^{-1} \ll \log y \ll \log x\ll P^{-\alpha}.$$
Therefore $$P_1^{-2^{ks}} \ll (P^{-\alpha})^{2^{ks}} \ll P^{-\frac{sk}{2}}.$$ So we have
$$L(x) \ll \frac{qP_{_{\mathcal{D}}}}{(xP_{_{\mathcal{D}}})^{\frac{k}{2}+1}}.$$ By integrating  both sides we deduce that $$\int_{P^{-1}_{0}}^{e^{P^{-\alpha/k}}} L(x)x^{\lambda -1}dx \ll \int_{P_{_{\mathcal{D}}}^{-1}}^{e^{P^{-\alpha/k}}} \frac{qP_{_{\mathcal{D}}}}{(xP_{_{\mathcal{D}}})^{\frac{k}{2}+1}}x^{\lambda -1}dx .$$ Since $\frac{k}{2}+1> \lambda $ we have
$$\int_{P_{_{\mathcal{D}}}^{-1}}^{P^{-\alpha/k}} L(x)x^{\lambda -1}dx \ll qP_{_{\mathcal{D}}}^{1-\lambda } \ll q(P^{s})^{1-\lambda }. $$ For larger $x$ we use Lemma \ref{th0},  which gives us that
$$M^{\mathcal{D}}_k(q, h) \ll qh^{k/2} P^{-2^{ks}+ks}.$$ Therefore we have
$$L(x) \ll \frac{qP^{-2^{ks}}}{x^{\frac{k}{2}+1}}$$ and
$$\int_{e^{P^{-\alpha/k}}}^{\infty} L(x)x^{\lambda -1}dx \ll qP^{-2^{ks}} \int_{e^{P^{-\alpha/k}}}^{\infty} \frac{x^{\lambda -1}}{x^{\frac{k}{2}+1}}dx$$ So taking $k= 2\lfloor \lambda\rfloor +1 $ implies that
$$\int_{e^{P^{-\alpha/k}}}^{\infty} L(x)x^{\lambda -1}dx  \ll q \frac{P^{-2^{ks}}}{e^{P^{-\alpha/k}}} \ll q(P^s)^{1-\lambda},$$ for $P^{-1}$ large enough, which finishes the proof.
\\

{\it E-mail address}: farzad.aryan@uleth.ca

\begin{thebibliography}{99}
\bibitem{C} H. Cramer,  On the order of magnitude of the difference between consecutive prime numbers. Acta Arithmetica.(1936): 23-46.
\bibitem{A} A. Granville,  Harold Cramer and the distribution of prime numbers. Scandanavian Actuarial. J.1995, no. 1, pages 12- 28.
\bibitem{er}P. Erd\H{o}s, The difference of consecutive primes. Duke Math. J. 6, (1940). 438--441
\bibitem{ho} C.Hooley, On the difference of consecutive numbers prime to n. Acta Arith. 8 1962/1963 343--347
\bibitem{sh} M. Hausman and H. Shapiro,  On the mean square distribution of primitive roots of unity. Comm. Pure Appl. Math. 26 (1973), 539--547.
\bibitem{mv} H. Montgomery and R. Vaughan,  On the distribution of reduced residues. Ann. of Math. (2) 123 (1986), no. 2, 311--333.
\bibitem{J.R} H. Halberstam and H. E. Richert, {\it Sieve Methods}, London Mathematical Society
monographs, No 4(London, New York: Academic Press, 1974).
\bibitem{MS} Hugh L. Montgomery and K. Soundararajan, Primes in short intervals, Comm. Math. Phys.
252 (2004), no. 1-3, 589-617.
\end{thebibliography}
\end{document}